\documentclass{amsart}

\usepackage{amssymb}
\usepackage{amsmath}
\usepackage{amsthm}
\usepackage{verbatim}
\usepackage{url}

\newcommand{\Z}{{\mathbb Z}}            
\newcommand{\C}{{\mathbb C}}            

\newcommand{\into}{\hookrightarrow}     

\DeclareMathOperator{\Res}{{\rm Res}}
\DeclareMathOperator{\Rec}{{\rm Rec}}

\newcommand{\F}{{\mathbf F}}                    

\newtheorem{theorem}{Theorem}[section]

\newtheorem{lemma}[theorem]{Lemma}
\newtheorem{prop}[theorem]{Proposition}

\theoremstyle{remark}
\newtheorem{remark}[theorem]{Remark}

\newcounter{listcounter1}

\begin{document}

\title{Reciprocity via Reciprocants}
\date{\today}

\author[Matthew Baker]{Matthew H. Baker}

\begin{abstract}
The determinant of a skew-symmetric matrix has a canonical square root given by the Pfaffian. Similarly, the resultant of two reciprocal polynomials of even degree has a canonical square root given by their {\em reciprocant}.
Computing the reciprocant of two cyclotomic polynomials yields a short and elegant proof of the Law of Quadratic Reciprocity.
\end{abstract}

\thanks{We thank Antoine Chambert-Loir for pointing us to M{\'e}rindol's paper \cite{Merindol}. Thanks also to Darij Grinberg, Franz Lemmermeyer, and Evan O'Dorney for helpful feedback on an earlier version of this paper.
The author was supported by NSF grant DMS-2154224 and a Simons Fellowship in Mathematics.}

\maketitle

\section{Introduction}

Let $p$ be a prime number and let $a$ be an integer not divisible by $p$.
The {\em Legendre symbol} $\left( \frac{a}{p} \right)$ is defined by $\left( \frac{a}{p} \right) = 1$ if $a$ is a square modulo $p$ and $\left( \frac{a}{p} \right) = -1$ otherwise.


According to {\em Euler's criterion}, $a^{(p-1)/2} \equiv 1 \pmod{p}$ if $\left( \frac{a}{p} \right) = 1$ and $a^{(p-1)/2} \equiv -1 \pmod{p}$ if $\left( \frac{a}{p} \right) = -1$.

The Law of Quadratic Reciprocity, first proved by Gauss, asserts that there is an unexpected relationship between $\left( \frac{p}{q} \right)$ and $\left( \frac{q}{p} \right)$ when $p,q$ are distinct odd primes, and a supplement to the law asserts that $\left( \frac{2}{p} \right)$ depends only on  $p$ modulo 8.

\begin{theorem}[Law of Quadratic Reciprocity] \label{thm:QR}
\mbox{}
\begin{enumerate}
\item[(a)] If $p$ and $q$ are distinct odd primes then $\left( \frac{p}{q} \right) \left( \frac{q}{p} \right) = (-1)^{\frac{p-1}{2}\frac{q-1}{2}}$.
\item[(b)] If $p$ is an odd prime then $\left( \frac{2}{p} \right) = (-1)^{\frac{p^2 - 1}{8}}$.
\end{enumerate}
\end{theorem}

There are currently more than 300 known proofs of the Law of Quadratic Reciprocity \cite{QRdatabase}. In this paper we will present an elegant proof that deserves to be better known. 
The basic approach, via the identity 
\begin{equation} \label{eq:introRES3}
{\rm Res}(g,f) = (-1)^{\deg(f) \cdot \deg(g)} {\rm Res}(f,g)
\end{equation}
for resultants, appears to have been independently discovered on at least two occasions \cite{Merindol,Hambleton-Scharaschkin}, see Section~\ref{sec:RelatedWork} below for a discussion of related work.

Our exposition is somewhat novel, in that a central role is played by an expression that we dub the {\em reciprocant}.\footnote{This is not a standard term; we chose the name both because it involves reciprocal polynomials and because of its relation to quadratic reciprocity.}
The resultant of two reciprocal\footnote{Reciprocal means that the coefficients read the same backwards and forwards; see Section~\ref{sec:recip} for more details.} polynomials $f$ and $g$ of even degree is always a square, and the reciprocant of $f$ and $g$ furnishes a canonical square root.
If $p$ and $q$ are distinct primes, the resultant of the cyclotomic polynomials $\Phi_p(x)$ and $\Phi_q(x)$ is always equal to 1, but their reciprocant $\Rec(\Phi_p(x),\Phi_q(x))$ turns out to be the Legendre symbol $\left( \frac{q}{p} \right)$.
By symmetry, we have $\Rec(\Phi_q(x),\Phi_p(x)) = \left( \frac{p}{q} \right)$, and part (a) of the Law of Quadratic Reciprocity is then a consequence of \eqref{eq:introRES3}. 

We also provide a proof via reciprocants of the supplementary law for $\left( \frac{2}{p} \right)$.

Our proof of the resultant identity $\Res(\Phi_p(x),\Phi_q(x))=1$ is original, to the best of our knowledge. It is in some ways more elementary than the other proofs we have seen of this formula. 

Throughout the article, we strive to keep the exposition as elementary as possible, with the goal of making the paper understandable by a reader who has taken basic undergraduate courses in number theory, abstract algebra, and linear algebra.
In order to make the paper as self-contained as possible, we provide two appendices, one on resultants and one on the trace polynomial (which is used to define the reciprocant).


\section{Resultants and Reciprocants}

\subsection{Resultants}

The resultant of two monic\footnote{We restrict ourselves to resultants of {\em monic} polynomials over an integral domain here, as (a) it's the only case we need and (b) the identities (RES1)-(RES4) look cleaner in the monic case.} polynomials $f,g \in R[x]$ over an integral domain $R$ satisfies numerous useful identities, including the following (see Appendix A for details):

\begin{itemize}
\item[(RES1)] If $f(x) = (x - \alpha_1)\cdots (x - \alpha_m)$ with all $\alpha_i$ in $R$, then ${\rm Res}(f,g) = \prod_{i} g(\alpha_i)$. 
\item[(RES2)] ${\rm Res}(g,f) = (-1)^{\deg(f) \cdot \deg(g)} {\rm Res}(f,g)$. 
\item[(RES3)] Suppose $\phi : R \to R'$ is a ring homomorphism. Then\footnote{Here $\phi(f) \in R'[x]$ denotes the image of $f \in R[x]$ under the homomorphism $R[x] \to R'[x]$ induced by $\phi$, and similarly for $\phi(g)$.} 
\[
\phi(\Res(f,g)) = \Res(\phi(f),\phi(g)).
\]
\item[(RES4)] If $g(x) = f(x) \cdot q(x) + r(x)$ with $f,g,r \in R[x]$ monic and $q \in R[x]$ arbitrary, then 
\[
{\rm Res}(f,g) = {\rm Res}(f,r).
\]
\end{itemize}


\subsection{Reciprocal polynomials and their traces} \label{sec:recip}

A polynomial $g(x) = a_0 + a_1 x + \cdots + a_{n} x^{n} \in R[x]$ with coefficients in a ring\footnote{All rings in this paper will be nonzero commutative rings with identity.} $R$ is called {\em reciprocal} if $a_n \neq 0$ and $a_k = a_{n-k}$ for all $k=0,1,\ldots,n$. Equivalently, $g$ is reciprocal if and only if $g(x) = x^{n}g(\frac{1}{x})$.

If $g \in R[x]$ is reciprocal of even degree $2m$, there is a unique polynomial $g^\#(x) \in R[x]$ of degree $m$ such that 
\begin{equation} \label{eq:Chebyshev transform}
g(x) = x^m g^\#(x + \frac{1}{x}).
\end{equation}
We call $g^\#(x)$ the {\em trace polynomial} of $g$ (see Appendix B for details). 
Note that if $g(x)$ is monic, then $g^\#(x)$ is monic as well.

The following lemma will be proved in Appendix B:

\begin{lemma} \label{lem:computing the reciprocant}
If $g(x) = \prod_{i=1}^m (x - \alpha_i)(x - \alpha_i^{-1})$ for some units $\alpha_1,\ldots,\alpha_m \in R^\times$, then $g$ is reciprocal and 
\begin{equation} \label{eq:product formula for gsharp}
g^\#(x) = \prod_{i=1}^m \left(x - (\alpha_i + \alpha_i^{-1}) \right).
\end{equation}
\end{lemma}

\begin{remark} \label{rem:split}
Conversely, it follows from \eqref{eq:Chebyshev transform} that if $K$ is a field, $g \in K[x]$ is reciprocal of even degree $2m$, and $L$ is a splitting field for $g$ over $K$, there exist $\alpha_1,\ldots,\alpha_m \in L^\times$ such that
$g(x) = \prod_{i=1}^m (x - \alpha_i)(x - \alpha_i^{-1})$.
\end{remark}

\subsection{Reciprocants}

Over an integral domain, the reciprocant is a canonical square root of the resultant of two reciprocal polynomials.
More precisely:

\begin{prop} \label{prop:res square}
If $R$ is an integral domain and $f,g \in R[x]$ are monic reciprocal polynomials of even degree, then
\[
\Res(f,g) = \Rec(f,g)^2,
\]
where $\Rec(f,g) := \Res(f^\#,g^\#) \in R$ is the {\em reciprocant} of $f$ and $g$.
\end{prop}

\begin{proof}
Let $K$ be the fraction field of $R$ and let $L$ be a splitting field for $f$ over $K$.
By Remark~\ref{rem:split}, we can write $f(x) = \prod_{i=1}^m (x - \alpha_i)(x - \alpha_i^{-1})$ with $\alpha_i \in L$ for all $i$. 
In what follows, will apply (RES3) to the natural injective map $\phi : R \to L$.

Let $a_i = \alpha_i + \alpha_i^{-1}$ for $i=1,\ldots,m$. We have:

\begin{align*}
\Res(f^\#,g^\#)^2 &= \prod_i g^\#(a_i) \cdot  \prod_i g^\#(a_i) \qquad\text{(by (RES1), (RES3), and \eqref{eq:product formula for gsharp})} \notag   \\
&=  \prod_i \alpha_i^{-m} g(\alpha_i) \cdot  \prod_i \alpha_i^m g(\alpha_i^{-1}) \qquad\text{(by \eqref{eq:Chebyshev transform})} \notag  \\
&= \prod_i g(\alpha_i) \cdot  \prod_i g(\alpha_i^{-1})  \notag  \\
&= \Res(f,g) \qquad\text{(by (RES1)).} \notag 
\end{align*}
\end{proof}

We will use the following in our proof of Quadratic Reciprocity:

\begin{prop} \label{prop:rec mod p}
If $g_1,g_2 ,h \in \Z[x]$ are monic and reciprocal polynomials of even degree and $n$ is a positive integer such that $g_1 \equiv g_2 \pmod{n}$, then
\[
\Rec(g_1,h) \equiv \Rec(g_2,h) \pmod{n}
\]
and 
\[
\Rec(h,g_1) \equiv \Rec(h,g_2) \pmod{n}.
\]
\end{prop}

A proof, based on property (RES3) of resultants, is given in Appendix B.

\section{Proof of the Law of Quadratic Reciprocity}

For $n \geq 1$, define 
\begin{equation} \label{eq:gn}
g_n(x) = \frac{x^n - 1}{x-1} = x^{n-1} + x^{n-2} + \cdots + x + 1 \in \Z[x].
\end{equation}

If $p$ is prime, then since $x^p - 1 \equiv (x-1)^p \pmod{p}$ we have 
\begin{equation} \label{eq:gpmodp}
g_p(x) \equiv (x-1)^{p-1} \pmod{p}.
\end{equation}

\begin{prop} \label{prop:Res g}
If $m,n$ are relatively prime positive integers, ${\rm Res}(g_m,g_n) = 1$.
\end{prop}

\begin{proof}
If $m=n=1$ then $\Res(g_m,g_n) = \Res(1,1) = 1$.
We may therefore suppose without loss of generality that $n > m$.
Note that since ${\rm gcd}(m,n)=1$, at least one of $m$ and $n$ is odd.

By the division algorithm, we can write $n = mq + r$ with $q,r$ integers such that $q \geq 0$ and $0 \leq r < q$.  
Since at least one of $m$ and $n$ is odd, the same is true for $m$ and $r$.

Working in the quotient ring $\Z[x]/(x^m - 1)$, we have

\begin{align*}
 x^n - 1
  &\equiv (x^m)^q \cdot x^r - 1 \notag\\
  &\equiv  1^q \cdot x^r - 1 \notag\\
  &\equiv  x^r - 1 \pmod{x^m - 1}. \notag\\
\end{align*}

In other words, there is a polynomial $h(x) \in \Z[x]$ such that 
\[
x^n - 1 = (x^m - 1)h(x) + x^r - 1.
\]

Dividing both sides by $x-1$ gives
\[
g_n(x) = g_m(x) h(x) + g_r(x).
\]

By (RES4) and (RES2), we have 
\begin{equation} \label{eq:euclidean}
\Res(g_m,g_n) = \Res(g_m,g_r) = \Res(g_r,g_m).
\end{equation}

Since ${\rm gcd}(m,n)=1$, it follows from \eqref{eq:euclidean} and the Euclidean algorithm that there is an integer $k\geq 1$ such that 
\[
\Res(g_m,g_n) = \Res(g_k,g_1) = \Res(g_k,1) = 1.
\]
\end{proof}

\begin{remark}
Conversely, if ${\rm gcd}(m,n) = d > 1$ then (RES1) and (RES3) (applied to the natural injection $\phi : \Z \into \C$) imply that $\Res(g_m,g_n)=0$, since a primitive $d^{\rm th}$ root of unity in $\C$ is a common root of $g_m$ and $g_n$.
\end{remark}

\begin{remark}
Here is an alternate proof of Proposition~\ref{prop:Res g} which is arguably more conceptual, but somewhat less elementary.
First, observe that if $K$ is any field and $\alpha \in K$ satisfies both $\alpha^m = 1$ and $\alpha^n = 1$, with ${\rm gcd}(m,n)=1$, then necessarily $\alpha = 1$.
Let $p$ be a prime number, let $\F_p$ be the finite field of order $p$, and let $\phi : \Z \to \F_p$ be the natural homomorphism.
Applying (RES3) to $\phi$ implies, together with (RES1) and the above observation with $K=\F_p$, that $\Res(g_m,g_n) \not\equiv 0 \pmod{p}$.
Since this holds for all prime numbers $p$, we must have $\Res(g_m,g_n) = \pm 1$.
By Proposition~\ref{prop:res square}, we must in fact have $\Res(g_m,g_n) = 1$.
\end{remark}

Assume from now on that $n$ is odd.
Since $g_n$ is a reciprocal polynomial of even degree, it follows from
\eqref{eq:Chebyshev transform} that
\begin{equation} \label{eq:gnsharp2}
g_n^\#(2) = g_n(1) = n.
\end{equation}

Furthermore, for any ring $R$, if $g(x)=(x-1)^{2m} \in R[x]$ then, by Lemma~\ref{lem:computing the reciprocant},
\begin{equation} \label{eq:special sharp}
g^\#(x) = (x-2)^m.
\end{equation}

\begin{remark} \label{rmk:recursion for g}
By Remark~\ref{rem:gnrecursion}, we have $g_1^\#(x) = 1$ and $g_3^\#(x) = x+1$, and
\begin{equation} \label{eq:gnrecursion} 
g_n^\#(x) = x g_{n-2}^\#(x) - g_{n-4}^\#(x)
\end{equation}
for all odd integers $n \geq 5$.
This implies that the polynomials $g_n^\#$ are related to the 
classical {\em Lucas polynomials} $L_n(x)$, defined for $n \geq 0$ by $L_0(x) = 2, L_1(x) = x$, and $L_n(x) = xL_{n-1}(x) + L_{n-2}(x)$, as follows. 
For $n\geq 1$ odd, define $H_n(x)$ by $L_n(x) = xH_n(x^2)$. Then $g_n^\#(x) = H_n(x-2)$.
\end{remark}

\begin{proof}[Proof of the Law of Quadratic Reciprocity]
Let $p,q$ be distinct odd primes. 

Since $\Res(g_p,g_q) = 1$ by Proposition~\ref{prop:Res g}, it follows from Proposition~\ref{prop:res square} that $\Rec(g_p,g_q) \in \{ \pm 1 \}$.
We compute the following congruences modulo $p$:

{\allowdisplaybreaks
\begin{align*}
  \Rec(g_p,g_q) 
  &\equiv \Rec((x-1)^{p-1},g_q)  \qquad\text{(by \eqref{eq:gpmodp} and Proposition~\ref{prop:rec mod p})} \notag\\
    &\equiv \Res((x-2)^{\frac{p-1}{2}},g^\#_q)  \qquad\text{(by \eqref{eq:special sharp} and the definition of the reciprocant)} \notag\\
  &= g_q^\#(2)^{\frac{p-1}{2}}   \qquad\text{(by (RES1))} \notag\\  
  &= q^{\frac{p-1}{2}}   \qquad\text{(by \eqref{eq:gnsharp2})} \notag\\
  &\equiv \left( \frac{q}{p} \right) 
  \qquad\text{(by Euler's criterion).} \notag
\end{align*}
}

Since $\Rec(g_p,g_q)$ and $\left( \frac{q}{p} \right)$ both belong to $\{ \pm 1 \}$, it follows that 
\begin{equation} \label{eq:legendre1}
\Rec(g_p,g_q) =  \left( \frac{q}{p} \right).
\end{equation}

By symmetry, we also have 
\begin{equation} \label{eq:legendre2}
\Rec(g_q,g_p) =  \left( \frac{p}{q} \right).
\end{equation}

The Law of Quadratic Reciprocity now follows from \eqref{eq:legendre1}, \eqref{eq:legendre2}, and (RES2).
\end{proof}

\section{The supplementary law}

We can use a similar argument to prove the supplementary law characterizing $\left( \frac{2}{p} \right)$ when $p$ is an odd prime.
Actually, it turns out to be more straightforward to establish a formula for $\left( \frac{-2}{p} \right)$.

By Euler's criterion, we have
\[
\left( \frac{-2}{p} \right) = \left( \frac{-1}{p} \right) \cdot \left( \frac{2}{p} \right) = (-1)^{\frac{p-1}{2}} \left( \frac{2}{p} \right),
\]
so the supplemental law is equivalent to:

\begin{theorem} \label{thm:QRsupplement}
If $p$ is an odd prime then $\left( \frac{-2}{p} \right) = 1$ if $p \equiv 1$ or $3 \pmod{8}$ and $\left( \frac{-2}{p} \right) = -1$ if $p \equiv 5$ or $7 \pmod{8}$.
\end{theorem}

Instead of the polynomial $g_2(x) = x+1$, we will use the cyclotomic polynomial $\Phi_4(x) = x^2 + 1$. Note that $\Phi_4$ is a reciprocal polynomial of even degree, with $\Phi_4^\#(x) = x$.

\begin{prop} \label{prop:Res 2}
If $n$ is an odd positive integer, then ${\rm Rec}(\Phi_4,g_n)$ is equal to 1 if $n$ is 1 or 3 (mod 8) and $-1$ if $n$ is 5 or 7 (mod 8).
\end{prop}

\begin{proof}
We have 
\[
\Rec(\Phi_4,g_n) = \Res(x, g_n^\#) = g_n^\#(0),
\]
so it suffices to evaluate $g_n^\#(0)$.

This is a straightforward but tedious calculation given \eqref{eq:Chebyshev transform}, which implies that 
\[
g_n^\#(0) = g_n(i) i^{-\frac{n-1}{2}} = \frac{i^n - 1}{i-1} i^{-\frac{n-1}{2}},
\]
where $i^2 = -1 \in \C$.

Alternatively, recall from Remark~\ref{rmk:recursion for g} that for $n \geq 5$ odd we have
\begin{equation*}
g_n^\#(x) = x g_{n-2}^\#(x) - g_{n-4}^\#(x).
\end{equation*}

From this, a simple inductive argument shows that  
$g_n^\#(0) = 1$ if $n$ is 1 or 3 (mod 8) and $-1$ if $n$ is 5 or 7 (mod 8).
\end{proof}

\begin{proof}[Proof of Theorem~\ref{thm:QRsupplement}]
Let $p$ be an odd prime. We compute:
{\allowdisplaybreaks
\begin{align*}
  \Rec(\Phi_4,g_p) 
  &\equiv \Rec(\Phi_4,(x-1)^{p-1})  \qquad\text{(by \eqref{eq:gpmodp} and Proposition~\ref{prop:rec mod p})} \notag\\
    &\equiv \Res(x,(x-2)^{\frac{p-1}{2}})  \qquad\text{(by \eqref{eq:special sharp} and the definition of the reciprocant)} \notag\\
  &= (-2)^{\frac{p-1}{2}}   \qquad\text{(by (RES1))} \notag\\  
  &\equiv \left( \frac{-2}{p} \right) \pmod{p}  \qquad\text{(by Euler's criterion).} \notag
\end{align*}
}

Since $\Rec(\Phi_4,g_p)$ and $\left( \frac{-2}{p} \right)$ both belong to $\{ \pm 1 \}$, we have
$\Rec(\Phi_4,g_p) =  \left( \frac{-2}{p} \right)$,
which implies the desired result via Proposition~\ref{prop:Res 2}.
\end{proof}

\section{Related work} \label{sec:RelatedWork}

The proof of Quadratic Reciprocity given here is closely related to several existing arguments. The earliest reference we're aware of for a proof of quadratic reciprocity based on resultants of cyclotomic polynomials is J.-Y.~M\'{e}rindol's paper \cite{Merindol}, which was published in an obscure French higher education journal called L'Ouvert. A similar proof appears to have been independently discovered by Hambleton and Scharaschkin in \cite{Hambleton-Scharaschkin}. We learned of the basic argument behind these papers from Antoine Chambert-Loir's blog post \cite{Chambert-Loir}.

The main new ingredient in the present paper is a systematic use of Proposition~\ref{prop:res square} and the quantity we've dubbed the reciprocant.
As far as we know, our arguments proving the supplemental law (Theorem~\ref{thm:QRsupplement}) are also new. 

Our treatment of resultants was inspired by a paper of Barnett \cite{Barnett}.
Our proof of Proposition~\ref{prop:Res g} makes use of the Euclidean algorithm and property (RES4) of resultants; 
this approach is also used, for example, in \cite{FHR}.

Although we have not seen Proposition~\ref{prop:res square} explicitly stated in a published paper, it is mentioned without proof in a Math Overflow post by Denis Serre \cite{Serre}.
The main ingredients in the proof of Proposition~\ref{prop:res square} are also contained in the proof of \cite[Theorem 3.4]{Loper-Werner}.

The first published work we're aware of that computes the resultant of two cyclotomic polynomials is F.~E.~Diederichsen's paper \cite{Diederichsen}. Diederichsen's results were extended, and his proofs simplified, in Apostol's paper \cite{Apostol}.
Some other papers computing resultants of Fibonacci--Lucas type polynomials include \cite{Merindol,Hambleton-Scharaschkin,Loper-Werner,FHR}.

As noted in \cite[Section 3]{Hambleton-Scharaschkin}, the key step underlying our proof of Quadratic Reciprocity, which is identifying the Legendre symbol with a resultant, is closely related to one of Eisenstein's classical proofs \cite[Chapter 8.1]{LemmermeyerBook}.
There are also close connections to the more recent proof of Swan \cite{Swan}.

The proof given in this paper is also closely related to the proof in the author's blog post \cite{BakerBlog}. 

A resultant-based approach to quadratic reciprocity in the function field case is given in \cite{Clark-Pollack}. See Section 3.4 of {\em loc.~cit.} for remarks about other proofs of the Law of Quadratic Reciprocity which ultimately boil down (either explicitly or in disguise) to property (RES2) of resultants.

\appendix
\section{Resultants} \label{AppendixA}

Let $R$ be a ring and let $f,g \in R[x]$ be monic polynomials.
Inspired by an observation of Barnett \cite{Barnett}, we define the {\em resultant} of $f$ and $g$ to be
\[
\Res(f,g) := {\rm det} \left( g(C_f) \right) \in R,
\]
where $C_f$ is the {\em companion matrix} of $f(x) = a_0 + a_1 x + \cdots + a_{m-1}x^{m-1} + x^m$:
\[
C_f :=
   \begin{pmatrix}
   0      &  0     & 0      & \cdots & 0 & -a_0  \\
   1      & 0      & 0      & \cdots & 0 & -a_1 \\
   0      & 1     & 0      & \cdots & 0 & -a_2 \\
   \vdots & \vdots & \vdots & \ddots & \vdots & \vdots \\
  0    & 0   &   0  & \cdots & 1 & -a_{m-1}  \\
  \end{pmatrix}
\]

We assume for the rest of this section that $R$ is an integral domain with fraction field $K$.
We will use the following two well-known facts from linear algebra:

\begin{itemize}
\item[(LA1)] The characteristic polynomial of $C_f$ over $K$ is $f$ (cf.~\cite[Lemma 8.4]{Liesen-Mehrmann}).
\item[(LA2)] If an $m \times m$ matrix $A$ over $K$ has characteristic polynomial $f$, and if $f$ factors over some extension field $L$ of $K$ as $f(x) = (x-\lambda_1)\cdots (x-\lambda_m)$, 
then the characteristic polynomial of $g(A)$ is $(x - g(\lambda_1))\cdots (x - g(\lambda_m))$. ({\bf Proof:} By \cite[Theorem 14.17]{Liesen-Mehrmann}, $A$ is similar to an upper triangular matrix $B$ with $\lambda_1,\ldots,\lambda_m$ on the diagonal. 
The diagonal entries of $g(B)$ are $g(\lambda_1),\ldots,g(\lambda_m)$, and by \cite[Theorem 8.12]{Liesen-Mehrmann} the characteristic polynomial of $p(A)$ is equal to that of $p(B)$.)
\end{itemize}

Assume $f$ splits into linear factors over $R$ as $f(x) = (x-\alpha_1)\cdots (x-\alpha_m)$.
Then by (LA1) and (LA2), the characteristic polynomial of $g(C_f)$ over $K$ is 
\[
(x - g(\alpha_1))\cdots (x - g(\alpha_m)).
\]

It follows that the determinant of $g(C_f)$ is $\prod_{i=1}^m g(\alpha_i)$, which proves (RES1). 
Moreover, if $g(x) = (x-\beta_1)\cdots (x-\beta_n)$ with all $\beta_j \in R$ then
\begin{equation} \label{eq:symmetric product for res}
\Res(f,g) = \prod_{i,j} (\alpha_i - \beta_j) \in R.
\end{equation}

If we view the coefficients of $f$ and $g$ as indeterminates, the expression ${\rm det} \left( g(C_f) \right)$ is a polynomial of degree $m+n$ with integer coefficients in these variables.
In other words, there is a multivariate polynomial $S_{m,n} \in \Z[x_0,x_1,\ldots,x_{m-1},y_0,y_1,\ldots,y_{n-1}]$ such that for every ring $R$ and every pair of monic polynomials 
$f(x) = a_0 + a_1 x + \cdots + a_{m-1} x^{m-1} + x^m$ and $g(x) = b_0 + b_1 x + \cdots + b_{n-1} x^{n-1} + x^n$ in $R[x]$, 
\[
{\rm Res}(f,g) = S_{m,n}(a_0,a_1,\ldots,a_{m-1},b_0,b_1,\ldots,b_{n-1}).
\]

The `functoriality' relation (RES3) follows easily from this observation.

By (RES3) and the fact that $R$ is an integral domain, we may replace $R$ by a splitting field $L$ for $fg$ over the fraction field $K$ of $R$.
The identity (RES2) then follows immediately from \eqref{eq:symmetric product for res}.

In the same way, we can reduce the proof of (RES4) to the case where $f(x) = (x-\alpha_1)\cdots (x-\alpha_m)$ with all $\alpha_i \in R$.
Using (RES1), we compute:

{\allowdisplaybreaks
\begin{align*}
{\rm Res}(f(x),r(x)) &=  \prod_{i=1}^m r(\alpha_i) \\
&= \prod_{i=1}^m \left( g(\alpha_i) - f(\alpha_i)q(\alpha_i) \right) \\
&= \prod_{i=1}^m g(\alpha_i) \\
&=  {\rm Res}(f(x),g(x)),
\end{align*}
}
which proves (RES4).

\section{The trace polynomial}

Our primary goal in this Appendix is to prove: 

\begin{prop}
\label{prop:hpoly2}
Suppose $R$ is a ring and $g \in R[x]$ is a reciprocal polynomial of even degree $2m$. 
Then there is a unique polynomial $h(x) \in R[x]$ of degree $m$ such that $g(x) = x^m h(x + \frac{1}{x})$.
\end{prop}

The following proof was suggested by Darij Grinberg.

\begin{proof} 
We first prove the existence of $h(x)$. This will be done by induction on $m$. The base case $m = 0$ is clear. For the induction step, let $g(x) = a_0 + a_1 x + ... + a_{2m} x^{2m}$ be a reciprocal polynomial of degree $2m$; in particular, $a_{2m} = a_0$. 
Thus $\tilde{g}(x) := \left( g(x) - a_0 (1 + x^2)^m \right) / x$ is a reciprocal polynomial of degree $2(m-1)$. By the inductive hypothesis, $\tilde{g}(x) = x^{m-1} \tilde{h}(x + 1/x)$ for some polynomial $\tilde{h}(x)$ of degree $m-1$ . 
Setting $h(x) = a_0 x^m + \tilde{h}(x)$ yields $g(x) = x^m h(x + 1/x)$, as desired. This establishes the existence of $h$.

The uniqueness of $h$ follows by reversing the existence argument. More formally, we again proceed by induction on $m$. 
The base case $m=0$ is obvious. For the induction step, note that the equation $g(x) = x^m h(x + 1/x)$ implies that the $x^m$-coefficient of $h(x)$ must be $a_0$. 
Let $\tilde{h}(x) = h(x) - a_0 x^m$, which has degree $m-1$, and let $\tilde{g}(x) = \left( g(x) -  a_0 (1 + x^2)^m \right) / x$, which is reciprocal of degree $2(m-1)$. Then $\tilde{g}(x) = x^{m-1} \tilde{h}(x + 1/x)$, and by the inductive hypothesis 
$\tilde{h}(x)$ is uniquely determined by $\tilde{g}(x)$. It follows that $h$ is uniquely determined by $g$.

\end{proof}

Following the terminology of \cite[\S{2.1}]{Gross-McMullen}, we define the {\em trace polynomial} $g^\#$ of $g$ to be the polynomial $h$ appearing in Proposition~\ref{prop:hpoly2}.

\begin{remark} \label{rem:Lemmermeyer}
The following alternative proof of the existence portion of Proposition~\ref{prop:hpoly2} was
suggested by Franz Lemmermeyer, and provides an explicit recursion which will be useful in the next remark.

Write $g(x) = a_0 + a_1 x + \cdots + a_{2m} x^{2m}$ with $a_i = a_{2m-i}$ for all $0 \leq i \leq m$ and $h(x) = b_0 + b_1 x + \cdots + b_m x^m$. 
We wish to prove that we can uniquely solve for the coefficients of $h$ in terms of the coefficients of $g$.

In the Laurent polynomial ring $R[x, \frac{1}{x}]$, we have the identity
\[
x^{-m}g(x) = a_0(x^m + x^{-m}) + a_1(x^{m-1} + x^{-(m-1)}) + \cdots + a_{m-1} (x + x^{-1}) + a_m,
\]
so it suffices to prove the result for the special Laurent polynomials $f_n(x) := x^n + x^{-n}$ for all $n \geq 0$.
In other words, we want to prove that for each $n \geq 0$, there is a polynomial $h_n(x) \in R[x]$ of degree $n$ such that $f_n(x) = h_n(x + x^{-1})$.

We prove existence of the polynomials $h_n(x)$ by induction on $n$. 
The result is trivial for $n=0,1$, so we may assume that $n\geq 2$ and that the result is true for polynomials of degree at most $n-1$.
A simple calculation gives
\[
f_n(x) = (x + x^{-1}) f_{n-1}(x) - f_{n-2}.
\]

Therefore, if we set $h_0(x) = 2$, $h_1(x) = x$, and 
\begin{equation} \label{eq:hnrecursion}
h_n(x) = xh_{n-1}(x) - h_{n-2}(x),
\end{equation}
we will have the desired identity $f_n(x) = h_n(x + x^{-1})$.
\end{remark}

\begin{remark} \label{rem:gnrecursion}
For $n\geq 0$, define $g_{2n+1}(x) = \sum_{k=0}^{2n} x^k$ as in \eqref{eq:gn}.

Then with $f_k(x)$ and $h_k(x)$ as in Remark~\ref{rem:Lemmermeyer}, for $n\geq 1$ we have $x^{-n} g_{2n+1} = 1 + \sum_{k=1}^n f_k(x)$, and thus
$g^\#_{2n+1}(x) = 1 + \sum_{k=1}^n h_k(x)$.

Since $g_1(x) = 1$ and $g_3(x)=1+x+x^2$, we have $g_1^\#(x) = 1$ and $g_3^\#(x) = x+1$.
Moreover, since $h_k(x) = x h_{k-1}(x) - h_{k-2}(x)$ for $k \geq 2$, it follows from \eqref{eq:hnrecursion} that for $n \geq 2$,
\[
x g^\#_{2n-1}(x) - g^\#_{2n-3}(x) = x + \sum_{k=1}^{n-1} \left( xh_k(x) - h_{k-1}(x) \right) - 1 + h_0 = 1 + x + \sum_{k=2}^n h_k(x) = g^\#_{2n+1}.
\]

In other words, for all odd integers $n \geq 5$ we have
\begin{equation}
g_n^\#(x) = x g_{n-2}^\#(x) - g_{n-4}^\#(x).
\end{equation}
\end{remark}

\begin{proof}[Proof of Lemma~\ref{lem:computing the reciprocant}]
To see that $g(x) = \prod_{i=1}^m (x - \alpha_i)(x - \alpha_i^{-1})$ is reciprocal, we compute: 
\begin{align*}
x^{2m}g(\frac{1}{x}) & = x^{2m} \prod (\frac{1}{x} - \alpha_i)(\frac{1}{x} - \frac{1}{\alpha_i}) \\
&= \prod (1 - \alpha_i x)(1 - \frac{1}{\alpha_i} x) \\
&= (-1)^m \prod \alpha_i (x  -  \frac{1}{\alpha_i}) \cdot (-1)^m  \prod \frac{1}{\alpha_i} (x - \alpha_i) \\
&= \prod (x  -  \frac{1}{\alpha_i})(x - \alpha_i) \\
&= g(x).
\end{align*}

To prove \eqref{eq:product formula for gsharp}, the case $m=1$ can be handled by a simple computation: setting $\alpha=\alpha_1$ and $a = \alpha + \alpha^{-1}$, we have $g(x) = (x - \alpha)(x - \alpha^{-1}) = x^2 - ax + 1 = x(x + \frac{1}{x} - a)$, and thus
$g^\#(x) = x - a$. The general case follows immediately from the special case $m=1$: if $a_j = \alpha_j + \alpha_j^{-1}$ then $g^\#(x) = \prod_{j=1}^m  (x - a_j)$. 
\end{proof}

As mentioned in the text, we define the {\em reciprocant} $\Rec(f,g)$ of two reciprocal polynomials of even degree to be $\Res(f^\#,g^\#)$.

\begin{proof}[Proof of Proposition~\ref{prop:rec mod p}]
It suffices, by (RES2), to prove the following statement: if $g_1,g_2 ,h \in \Z[x]$ are reciprocal of even degree and $n$ is a positive integer such that $g_1 \equiv g_2 \pmod{n}$, then $\Rec(g_1,h) \equiv \Rec(g_2,h) \pmod{n}$.

By Proposition~\ref{prop:hpoly2}, we have $g_1^\#(x) \equiv g_2^\#(x) \pmod{n}$. Applying (RES3) to the natural ring homomorphism $\phi : \Z \to \Z/n\Z$ shows that $\Res(g^\#_1,h^\#) \equiv \Res(g^\#_2,h^\#) \pmod{n}$ as desired.
\end{proof}

\begin{small}
 \bibliographystyle{plain}
 \bibliography{QR}
\end{small}

\end{document}